\documentclass[dvips,17pt]{article}

\usepackage{amsmath,amssymb,amsfonts,amsthm,amsbsy}
\usepackage{comment,cite,color}
\usepackage{graphicx,cite}
\usepackage{pstricks}
\usepackage{psfrag}
\usepackage{color}
\usepackage{subfigure}
\usepackage{nextpage}
\usepackage{layout}
\usepackage{enumerate}
\usepackage{amsthm}
\usepackage{dsfont}
\usepackage{amssymb}
\usepackage{subfigure}

\newcommand{\abs}[1]{\lvert#1\rvert}
\newtheorem{theorem}{Theorem}[section]
\newtheorem{proposition}[theorem]{Proposition}

\newtheorem{lemma}[theorem]{Lemma}
\newtheorem{remark}[theorem]{Remark}

\newcommand{\argmin}[1]{\mathop{\rm argmin}\limits_{#1}}
\newcommand{\innerp}[1]{\langle {#1} \rangle}

\def\c{\mathbf c}

\def\Z{\mathbb{ Z}}

\def\R{\mathbb{ R}}
\def\N{\mathbb{ N}}
\def\C{\mathbb{ C}}

\begin{document}

\title{
On  sparse interpolation  and the design of   deterministic interpolation points   }

\author{
Zhiqiang Xu
\thanks{Institute of Computational Mathematics and Scientific/Engineering Computing, AMSS, the Chinese Academy of Sciences, Beijing, China,
Email: {xuzq@lsec.cc.ac.cn}}
\and
Tao Zhou
\thanks{Institute of Computational Mathematics and Scientific/Engineering Computing, AMSS, the Chinese Academy of Sciences, Beijing, China,
Email: {tzhou@lsec.cc.ac.cn}.
}}

\date{}

\maketitle

\begin{abstract}
In this paper, we build up a framework for  sparse interpolation.
We first investigate the theoretical limit of the number of unisolvent points for  sparse interpolation under a general setting
and try to answer some basic questions of this topic. We also explore
the relation between  classical interpolation and  sparse interpolation. We second consider the design of the interpolation points for the $s$-sparse functions in high dimensional Chebyshev bases, for which the possible applications include uncertainty quantification, numerically solving stochastic or parametric PDEs and  compressed sensing.  Unlike the traditional random sampling method, we present in this paper a deterministic method to produce the interpolation points, and show its  performance with $\ell_1$ minimization by analyzing the
mutual incoherence of the interpolation matrix. Numerical experiments show that the deterministic points have a similar
performance with that of the random points.
\end{abstract}

\section{Introduction}

In signal processing, computer algebra, as well as in uncertainty quantification, there are increasing needs to efficiently recover a function  from a rather
small set of function values, where the function has sparse representations in some bases.
We state the problem as follows.
Assume that $\Omega\subset \R^d$ and that $\{B_j\}_{j\in \Lambda}$ is a set of $N$ complex-valued functions defined on $\Omega$,
where $\Lambda$ is an index set with $N:=\# \Lambda.$
A function
$$
f\,\,=\,\, \sum_{j\in \Lambda}c_jB_j
$$
is called {\em $s$-sparse} with respect to  $\{B_j\}_{j\in \Lambda}$ if at most $s\ll N$ coefficients of $\{c_j\}_{j\in \Lambda}$ are nonzero.
We denote by ${\bf U}^s$ the set of $s$-sparse functions with respect to  $\{B_j\}_{j\in \Lambda}$, i.e.,
\begin{equation}\label{eq:definition-U}
{\bf U}^s\,\,:=\,\,{\bf U}^s(\{B_j\}_{j\in \Lambda})\,\,:=\,\, \{ f =\sum_{j\in T} c_jB_j : T\subset \Lambda, \#T\leq s\}.
\end{equation}
The {\em $s$-sparse interpolation} with the functions $\{B_j\}_{j\in \Lambda}$ and the domain $\Omega$ is to  reconstruct  $f\in {\mathbf U}^s$ from $m$ samples $\{x_j,f(x_j)\}_{j=1}^m$ where $\{x_1,\ldots,x_m\}\subset \Omega$ are $m$ distinct points. In other words, one wants to find an index set
 $T_0\subset \Lambda$ with $\#T_0\leq s$  and coefficients $\{c_j\}_{j\in T_0}$, such that
 \begin{equation}\label{eq:int}
 \sum_{j\in T_0} c_jB_j(x_j)=f(x_j),\quad j=1,\ldots,m.
 \end{equation}
The point set $\{x_1,\ldots,x_m\}$ is said to be {\em  unisolvent}  if
$$
f(x_j)=g(x_j),\qquad j=1,\ldots,m
$$
implies $f\equiv g$ whenever  $f,g\in {\mathbf U}^s$.

\subsection{Related work}
Compressed sensing presents a  theoretical framework for investigating the $s$-sparse interpolation.
Let $\mathbf{A}\in \C^{m\times N}$ be the interpolation matrix, namely,
\begin{equation}\label{eq:A}
\mathbf{A}\,:=\,\,[B_{j}(x_t)]_{ t=1,\ldots,m,\, j\in \Lambda}
\end{equation}
and set
$$
\mathbf{b}\,:=\,\,[f(x_1),\ldots,f(x_m)]^\top.
$$
The $s$-sparse interpolation is equivalent to find a $s$-sparse vector in the following set
$$
\{\c\in \C^N: \mathbf{A}\c=\mathbf{b}\}.
$$
Based on compressed sensing theory, one can use the $\ell_1$ minimization to find the sparse solution to $\mathbf{Ac}=\mathbf{b}$ provided that
$\mathbf{A}$ satisfies the RIP condition \cite{CRT}.
In fact, let
\begin{equation}\label{eq:P1}
(P_1)\qquad\quad \c^\#:=\argmin{ \c\in \C^N}  \{\|\c\|_1\,\, \text{\rm subject to}\,\, \mathbf{A}\c=\mathbf{b}\}.
\end{equation}
Then one can reconstruct
$
f=\sum_{j}\c^{\#}_j B_j
$
successfully  in many settings for $\mathbf{A}$. Thus, we  can employ the methods in compressed sensing to investigate the $s$-sparse interpolation.

We next review  results for some special bases $\{B_j\}_{j\in \Lambda}$ and $\Omega$, which are obtained using  techniques developed in compressed sensing.
When  $B_j(t)=\exp(2\pi i j  t)$ (where $i^2=-1$) and $\Omega=[0,1]$, the $s$-sparse interpolation is reduced to the recovery of  sparse
trigonometric polynomials, which is an active research topic in recent years \cite{Rauhut1 , Rauhut2, xu, Iwen, LWC}.
In this direction, it is shown  that one can recover $f\in {\mathbf U}^s$ via $\ell_1$  minimization  from $\{x_j,f(x_j)\}_{j=1}^m$ with high probability when $x_1,\ldots,x_m$ are $m\asymp s(\log N)^4$ random points in $[0,1]$.
In the area of uncertainty quantification \cite{Fabio,XiuK,SPDE}, one is interested in the cases where $\{B_j\}_{j\in \Lambda}$ are orthogonal polynomials, such as Legendre polynomials and Chebyshev polynomials etc.
So, in \cite{RW}, one investigates the case with $B_j(t)$ being the $j$th order Legendre polynomials defined on $\Omega=[-1,1],$ and shows that $m\asymp s\log^4N$ sampling points (which are chosen randomly according to the Chebyshev measure) are enough to recover $s$-sparse Legendre polynomials with high probability. In \cite{XiuL1}, this result is extended to the high dimensional cases with $B_j$ being the tensor product of one dimensional Legendre polynomials. Some key properties, such as the RIP property of the interpolation matrix,  are also investigated in the above literatures. We remark that in the above works, the interpolation  points are all chosen by a random method.
Deterministic sampling is also investigated  for sparse trigonometric polynomials by taking the advantage of the structure of $\exp(2\pi i j t)$ \cite{cprony, LWC, Iwen}. In this direction, the classic  Prony method is also extended to investigate the sparse interpolation with the one dimensional Chebyshev polynomial bases \cite{prony}.

\subsection{Our contribution}
The aim of this paper is twofold. Despite many literatures on  sparse interpolation, there is little work on the theoretical limit of the number of  unisolvent points.  We first  provide a framework of sparse interpolation. Particularly, given $\{B_j\}_{j\in \Lambda}$ and $\Omega$,
we are interested in the following problems:
  \begin{enumerate}[{\bf Problem} 1.]
  \item What is the minimum $m$ for which there exists  a point set  $\{x_1,\ldots,x_m\}$  such that it is unisolvent for the $s$-sparse
  interpolation?
  \item What is the minimum $m$ for which  any point set  $\{x_1,\ldots,x_m\}$  with $m$ distinct points is unisolvent for the $s$-sparse   interpolation?
  \end{enumerate}
In this work, we shall employ the results in classical approximation theory to investigate Problem  1, 2, and bridge a gap between the $s$-sparse interpolation and the classical interpolation.

Our second aim is to design a point set  $\{x_1,\ldots,x_m\}$ such that one can recover efficiently the $s$-sparse function $f$ from
$\{(x_j,f(x_j))\}_{j=1}^m.$ We state the problem as follows:
\begin{enumerate}[{\bf Problem} 3.]
  \item How to choose a point set $\{x_1,\ldots,x_m\}$ such that one can recover $f\in {\mathbf U}^s$ efficiently from $\{(x_j,f(x_j))\}_{j=1}^m$?
\end{enumerate}
Our original motivation for this work was the recovery of sparse multivariate Chebyshev polynomials which is
raised in uncertainty quantification  \cite{SPDE,spde2, XiuL1}.
So, for Problem 3, we focus on the case where $\{B_j\}_{j\in \Lambda}$ are high dimensional Chebyshev polynomials.
Motivated by the results in compressed sensing \cite{xu}, we  present a \textit{deterministic} method to produce the points $x_1,\ldots,x_m$,
 and show that the mutual incoherence constant of the interpolation matrix associated with the deterministic points is small.
Hence, one can recover the high dimensional  $s$-sparse Chebyshev polynomials by $\ell_1$ minimization (see Theorem \ref{th:TP}, \ref{th:TD}).
The numerical experiments show that the performance of our method is similar with that of  the random one.
We believe that our deterministic points have potential applications in  the context of uncertainty quantification, especially in  numerical solving stochastic PDEs.  The last, but not the least,  the interpolation points presented in this paper are in an analytic form and hence
they are easy to be produced.

This rest of the paper is organized as follows. Section 2 provides some necessary concepts
and results to be used in our investigation. We study the number of unisolvent points in Section 3.
In Section 4, we present  deterministic points for the $s$-sparse interpolation in high dimensional Chebyshev bases and analyze its recovery ability. In the last section, numerical experiments are given to show the  efficiency of the deterministic points.

\section{Preliminaries}
In this section, we introduce some preliminaries which will play important roles in the following sections.

{\bf Chebyshev systems.}  A system  of functions $\{f_1,\ldots,f_m\}$
defined on $\Omega\subset \R^d$ is called a {\em Chebyshev system} if the determinant
$$
\det [f_j(x_k)]_{1 \leq k,j \leq m}
$$
does not vanish for any $m$ distinct points $\{x_1,\ldots,x_m\}\subset \Omega$. Chebyshev systems are important in  many areas  \cite{AT}, such as approximation theory,  moment problems, etc.

{\bf Mutual Incoherence Constant.}
We suppose that $\mathbf{A}$ is a $m\times N$ matrix with column vectors $\mathbf{A}_1,\ldots,\mathbf{A}_N$. The mutual incoherence constant  of $\mathbf{A}$ is defined
as
\begin{equation}\label{eq:MIC}
\mu(\mathbf{A})\,\,:=\,\, \max_{k\neq j}\frac{\abs{\innerp{\mathbf{A}_k, \mathbf{A}_j}}}{\|\mathbf{A}_k\|_2 \cdot \|\mathbf{A}_j\|_2}.
\end{equation}
Assume that $\mathbf{c}_0$ is a $s$-sparse vector in $\mathbb{C}^N$. Then, if
\begin{equation}\label{eq:BPMIC}
\mu \,\,<\,\,\frac{1}{2s-1},
\end{equation}
the solution to $(P_1)$ with $\mathbf{b}=\mathbf{A}\mathbf{c}_0$ is exactly $\mathbf{c}_0$, i.e.,
$$
 \mathbf{c}_0=\argmin{\mathbf{c}\in \C^N} \{\|\mathbf{c}\|_1\,\, \text{\rm subject to}\,\, \mathbf{A}\mathbf{c}=\mathbf{A}\mathbf{c}_0\}.
$$
The result was first presented in \cite{DonHuo} for the case with $\mathbf{A}$ being the union of two orthogonal matrices, and was extended to general matrices by Fuchs \cite{Fuchs} and Gribonval \& Nielsen \cite{GrNi}.  In \cite{CaWaXu}, Cai, Wang and Xu  show that
 $\mu < \frac{1}{2s-1}$ is sufficient for stably  approximating   $\mathbf{c}$ in the noisy case.

{\bf Restricted Isometry Property. (cf. \cite{CanTao05})}
We say that $\mathbf{A}\in \C^{m\times N}$ satisfies the Restricted Isometry Property (RIP) of order $s$ with constant $\delta_s
\in [0,1)$ if
\begin{equation}\label{eq:con}
(1-\delta_s) \|\mathbf{c}\|_2^2 \leq \|\mathbf{A} \mathbf{c} \|_2^2 \leq (1+\delta_s) \|\mathbf{c}\|_2^2
\end{equation}
holds for all vectors $\mathbf{c}\in \C^N$ with $\|\mathbf{c}\|_0\leq s,$ where $\|\mathbf{c}\|_0$ denotes the number of nonzero entries of $\mathbf{c}$.  In fact, (\ref{eq:con}) is equivalent to require that the
Grammian matrices $\mathbf{A}_T^\top\mathbf{A}_T$ has all of its eigenvalues in $[1-\delta_s,
1+\delta_s]$ for all $T$ with $\# T\leq s$, where $\mathbf{A}_T$ denotes the submatrix of $\mathbf{A}$ whose columns are those with indexes in $T$.

In \cite{candes,CRT}, it is shown that,  with certain RIP constants $\delta_s$, such as $\delta_{s}< 1/3$ \cite{CZ}, the vector
 $\mathbf{c}_0\in \C^N$ can be recovered in the following sense
 \begin{equation}\label{eq:bestk}
 \|\mathbf{c}^\#-\mathbf{c}_0\|_2\,\,\lesssim\,\, \frac{\sigma_{s,1}(\mathbf{c}_0)}{\sqrt{s}},
 \end{equation}
 where $\mathbf{c}^\#$ is given by the $\ell_1$-minimization problem $(P_1)$ with the vector $\mathbf{b}=\mathbf{A}\mathbf{c}_0$ for the given $\mathbf{c}_0\in \C^N$
 and
 \begin{equation}\label{eq:besterror}
 {\sigma_{s,1}(\mathbf{c}_0)}\,\,:=\,\, \min_{\mathbf{c}\in \C^N, \|\mathbf{c}\|_0\leq s}\|\mathbf{c}_0-\mathbf{c}\|_1.
 \end{equation}

{\bf Weil's Exponential Sum Theorem (cf. \cite{weil}).}
Suppose that $p$ is a prime number. Let $f(x)=m_1x+m_2x^2+\cdots+m_dx^d,$ and assume that there is a $j, \,1\leq j\leq d,$ such that $p\nmid m_j,$ then
\begin{align*}
\left| \sum_{x=1}^p e^{\frac{2\pi {\rm i} f(x)}{p}} \right|\,\,\leq\,\, (d-1)\sqrt{p}.
\end{align*}

{\bf High Dimensional Chebyshev Polynomials.}
We denote by $\phi_{n_j}(x_j)$ the one dimensional
$n_j$th order Chebyshev polynomial with respect to the variable $x_j$, i.e.,
\begin{align*}
\phi_{n_j}(x_j)\,\,=\,\,\cos(n_j\cdot \mathrm{arcos}(x_j)), \qquad x_j\in [-1,1].
\end{align*}
High dimensional Chebyshev polynomials can be constructed by tensorizing the one-dimensional polynomials. To do this,
let us first define the following multi-index:
\begin{align*}
 \mathbf{n}=(n_1,\ldots,n_d)\in \mathbb{N}^d, \quad \mathrm{with} \quad  |\mathbf{n}|=n_1+\cdot\cdot\cdot+n_d.
\end{align*}
With such definitions, every $d$-dimensional Chebyshev
polynomial in multi-variate $\mathbf{x}=(x_1,x_2,\ldots,x_d)$ can be written as
\begin{align*}
 \mathbf{\Phi}_\mathbf{n}(\mathbf{x})= \prod_{j=1}^d \phi_{n_j}(x_j).
\end{align*}
Given $q, d\in \N,$  we define the following index sets
$$
\Lambda_{\bf P}^{q,d}:=\{{\bf n}=(n_1,\ldots,n_d)\in \N^d:\max_{j=1,\ldots,d}n_j\leq q \},
$$
and
$$
\Lambda_{\bf D}^{q,d}:=\{{\bf n}=(n_1,\ldots,n_d)\in \N^d:\abs{{\bf n}}\leq q \}.
$$
Under the above notions, the traditional full {\em tensor product} (TP) space yields
\begin{align*}
\mathbf{P}_q^d \,\,:=\,\, {\rm span} \big\{\mathbf{\Phi}_\mathbf{n}(\mathbf{x}):  \mathbf{n}\in \Lambda_{\bf P}^{q,d} \big\}.
\end{align*}
That is, one requires in  $\mathbf{P}_q^d$ that the polynomial degree in each
variable less than or equal to $q.$ A simple observation is that  the dimension of $\mathbf{P}_q^d$ is
\begin{align*}
{\rm dim}(\mathbf{P}^d_q)=\#\Lambda_{\bf P}^{q,d}=(q+1)^d.
\end{align*}
Note that when $d\gg 1$ the dimension of TP polynomial spaces grows very fast with the polynomial degree $q$, which  is the so-called \textit{curse of dimensionality}. Thus, the TP spaces are rarely used in practice provided  $d \geq 5$ (see also \cite{XiuL1}). Therefore,
when $d$ is large, the following \textit{total degree} (TD)  polynomial space is often used instead of the TP space \cite{Fabio,XiuL1}
\begin{align*}
\mathbf{D}^d_q \,\,:=\,\, {\rm span} \big\{\mathbf{\Phi}_\mathbf{n}(\mathbf{x}):  \mathbf{n}\in \Lambda_{\bf D}^{q,d} \big\}.
\end{align*}
The dimension of $\mathbf{D}^d_q $ is
\begin{align*}
{\rm dim}(\mathbf{D}^d_q)=\#\Lambda_{\bf D}^{q,d}={q+d\choose d}.
\end{align*}
Note that, when $d\geq 2$,
$$
{q+d\choose d}=\left(\frac{q}{d}+1\right)\cdot\left(\frac{q}{d-1}+1\right)\cdots (q+1) < (q+1)^d.
$$
Hence, the growth of the dimension of  $\mathbf{D}_q^d$ is much slower than that of $\mathbf{P}_q^d$.

\section{The number of  unisolvent points for  sparse interpolation }

In this section, we focus on the theoretical limit of the number of  unisolvent points for  sparse interpolation.  Particularly,
 we try to give solutions to Problem 1, 2.
 Based on compressed sensing theory, the point set  $\{x_1,\ldots,x_m\}$ is   unisolvent for the
  $s$-sparse interpolation if and only if any $2s$ columns of the interpolation matrix $\mathbf{A}\in \C^{m\times N}$ are linearly independent.  And hence, a  lower bound is $m\geq 2s$.  In Section 3.1, we investigate Problem 1 and show that  the lower bound  $2s$ is sharp  provided that any $2s$ functions   in $\{B_j\}_{j\in \Lambda}$ are strongly linearly independent. We study Problem 2 in Section 3.2 and show that any $2s$
  distinct points in $\Omega$ are unisolvent if any $2s$ functions in  $\{B_j\}_{j=1}^N$ are a Chebyshev system.

\subsection{The lower bound $2s$ is sharp}

To this end, we first introduce the definition of strongly  linearly independent on $\Omega_0\subset \Omega$.
Given $\c=(c_1,\ldots,c_k)\in \C^k$ and $k$ functions $f_1,\ldots,f_k$ defined on $\Omega_0$, set
$$
{\mathcal I}_\c\,\,:=\,\, {\mathcal I}_\c(\Omega_0, f_1,\ldots,f_k)\,\,:=\,\,\{x\in \Omega_0 : \sum_{t=1}^k c_tf_{t}(x)=0 \}
$$
and denote by $\lambda^*_d$  the $d$-dimensional  Lebesgue outer measure.
We say that the $k$ functions $f_{1},\ldots,f_{k}$
are {\em strongly  linearly independent on $\Omega_0$} if
$$
\lambda^*_d({\mathcal I}_\c)\,\,>\,\,0
$$
implies that
$$
c_1=c_2=\cdots=c_k=0.
$$
To state the following lemma, we view
$$
{\bf x}\,=\,(x_1,\ldots,x_k)\,\in\, \underbrace{\Omega_0\times\cdots\times \Omega_0}_k\,=\,\Omega_0^k
$$
as a point in $\R^{d\cdot k}$.
Then, we  have
\begin{lemma}\label{le:strongly}
The following properties are equivalent:
\begin{enumerate}[{\rm (i)}]
\item The functions $f_1,\ldots,f_k$ are strongly linearly independent on $\Omega_0$.
\item For ${\bf x}:=(x_1,\ldots,x_k)\in \Omega_0^k$, set
$$
{\mathcal S}:=\{{\bf x}:=(x_1,\ldots,x_k)\in \Omega_0^k: \det( \mathbf{A}_{\bf x})=0\}\,\subset\, \R^{d\cdot k} ,
$$
where
$$
\mathbf{A}_{\bf x}:=[f_t(x_j)]_{j=1,\ldots,k;\, t=1,\ldots,k}.
$$
Then
$$
\lambda_{d\cdot k}^*({\mathcal S})=0.
$$
\end{enumerate}
\end{lemma}
\begin{proof}
We first show (i) implies (ii). Suppose (ii) false, i.e.,
 $$
\lambda_{d\cdot k}^*({\mathcal S})>0,
$$
 and we shall derive a contradiction.  To state conveniently, we denote by  $\mathbf{A}_j$ the $j$th row of $\mathbf{A}_{\bf x}$.
 Let us keep in the mind that $\mathbf{A}_j$ only depends on the $x_j$.
 Then there exists a set
 $
 {\mathcal S'} \subset {\mathcal S}
 $
 with $\lambda_{d\cdot k}^*(S')>0$ and an integer, say $k$, such that when ${\bf x}\in {\mathcal S'}$, we have
 $$
 \mathbf{A}_{k}\in {\rm span}\{\mathbf{A}_1,\ldots,\mathbf{A}_{k-1}\}.
 $$
Hence,  the solution to
\begin{equation}\label{eq:sol}
\mathbf{A}_{\bf x} \c=0
\end{equation}
 is independent with $x_k$.
Note that $\lambda_{d\cdot k}^*({\mathcal S'})>0,$ and thus there exists a fixed $(x'_1,\ldots,x'_{k-1})\in \Omega_0^{k-1}$ such that
\begin{equation}\label{eq:xk}
\lambda_d^*(\{x_k\in \Omega_0 : (x'_1,\ldots,x'_{k-1},x_k)\in {\mathcal S'}\})\,\,>\,\,0.
\end{equation}
We  choose ${\bf x}:=(x'_1,\ldots,x'_{k-1},x_k)$ in (\ref{eq:sol}) and  take a non-zero  solution, say $\c^\#=(c_1^\#,\ldots,c_k^\#)$, which is independent with $x_k$ and only depends on $x'_1,\ldots,x'_{k-1}$. A simple observation is that, for any $x_k$ with $(x'_1,\ldots,x'_{k-1},x_k)\in {\mathcal S'}$ one has
$$
\sum_{t=1}^k c_t^\#f_{t}(x_k)=0,
$$
which implies that
\begin{equation}\label{eq:sub}
\{x_k\in \Omega_0 : (x'_1,\ldots,x'_{k-1},x_k)\in {\mathcal S'}\}\subset {\mathcal I}_{\c^\#}.
\end{equation}
Combining (\ref{eq:xk}) and (\ref{eq:sub}), we have
$$
\lambda_d^*({\mathcal I}_{\c^\#})>0,
$$
where $\c^\#$ is a non-zero vector.
This leads to a contradiction by the definition of strongly linearly independent.

We next show (ii) implies (i). Suppose (i) false, namely, there exists a non-zero vector  $\c$ such that
$$
\lambda_d^*({\mathcal I}_\c)>0.
$$
A simple observation is that
$$
{\mathcal I}_\c^k=\underbrace{{\mathcal I}_\c\times\cdots\times {\mathcal I}_\c}_k \subset {\mathcal S},
$$
which implies that
$$
\lambda_{d\cdot k}^*({\mathcal S})\,\,\geq\,\, \lambda_{d\cdot k}^*({\mathcal I}^k_\c)\,\,>\,\,0.
$$
This again leads to a contradiction.
\end{proof}
We are now ready to give the following theorem
\begin{theorem}\label{th:uni}
Suppose that $s\leq N/2$ and that there exits $\Omega_0\subset \Omega\subset \R^d$ such that any $2s$ functions in $\{B_j\}_{j\in \Lambda}$ are strongly  linearly independent on $\Omega_0$. Then there exist  $2s$ points $\{x_1,\ldots,x_{2s}\}\subset \Omega_0$
such that they are unisolvent for the $s$-sparse interpolation with the basis functions $\{B_j\}_{j\in \Lambda}$ and the domain $\Omega$.
\end{theorem}
\begin{proof}
We consider the interpolation matrix
$$
\mathbf{A}=[B_{j}(x_k)]_{ 1\leq k\leq 2s,\, j\in \Lambda}.
$$
To this end, we only need prove that $\det(\mathbf{A}_T)\neq 0$ for any $T\subset \Lambda$ with $\#T=2s$,
where $\mathbf{A}_T$ denotes the submatrix of $\mathbf{A}$ whose columns are those with indexes in $T$.
Set
$$
{\mathcal S}_T:=\{(x_1,\ldots,x_{2s})\in \Omega_0^{2s}: \det (\mathbf{A}_T)=0\}.
$$
Since $\{B_j\}_{j\in T}$ are strongly linearly independent on $\Omega_0$,  by Lemma \ref{le:strongly}, we have
$$
\lambda_{2\cdot d \cdot s}^*({\mathcal S}_T)=0.
$$
To state conveniently, set
$$
\overline{\mathcal S}\,\,:=\,\,\bigcup_{\scriptstyle T\subset\{1,\ldots,N\}\atop\scriptstyle \#T=2s} {\mathcal S}_T.
$$
Noting that $\lambda_{2\cdot d\cdot s}^*(\overline{\mathcal S})=0$, we have
$$
\lambda_{2\cdot d \cdot s}^*\left(\Omega_0\setminus \overline{\mathcal S}\right)\,\,>\,\,0,
$$
which implies that $\Omega_0\setminus\overline{\mathcal S}$  is a nonempty set.
Therefore, we can choose
$$
\{x_1,\ldots,x_{2s}\}\subset \Omega_0\setminus \overline{\mathcal S},
$$
such that
$$
\det (\mathbf{A}_T)\neq 0, \quad {\rm for}\,\, {\rm all}\quad T\subset\{1,\ldots,N\}, \,\, \#T=2s,
$$
and this completes the proof.
\end{proof}
\begin{remark}
Suppose that $\Omega=[-1,1]$ and $B_j$ is a polynomial with degree $j$ on $\Omega$. Then, a simple observation is that any $2s$ functions
in  $\{B_j\}_{j=1}^N$ are strongly linearly independent on $\Omega$. Then, Theorem \ref{th:uni} implies that there are $2s$ points $\{x_1,\ldots,x_{2s}\}\subset \Omega$
such that they are unisolvent for the $s$-sparse interpolation with the basis functions $\{B_j\}_{j=1}^N$ and the domain $\Omega$.
In particular, if one takes $B_j$ as the one dimensional Chebyshev polynomial with degree $j$, the result implies that $2s$ points can be used to
theoretically recover $s$-sparse Chebyshev polynomial, which implies the result in \cite{prony}.
\end{remark}

We next show that, without the condition of strongly  linearly independent on some $\Omega_0$, it is possible that any  $2s$ points are not unisolvent. We consider the sparse interpolation for B-spline basis which is useful in  signal processing \cite{Bspline}. The first order B-spline on the interval $[j,j+1)$ is defined by
\begin{eqnarray*}
B_{1,j}(x)&:=& \begin{cases}
1, & j\leq x<{j+1},\\
0, & {\rm otherwise.}
\end{cases}
\end{eqnarray*}
 Then we consider the $s$-sparse interpolation for the basis functions $\{B_{1,j}\}_{j=1}^N$ and the
domain $\Omega=[1,N+1]$. A simple argument shows that $\{B_{1,j}\}_{j=1}^N$ are not strongly linearly independent on
any $\Omega_0\subset\Omega$.  Then, we have
\begin{proposition}
Suppose $s=1$. If the points $\{x_1,\ldots,x_m\}\subset \Omega$ are unisolvent for the $s$-sparse interpolation with the basis $\{B_{1,1},\ldots,B_{1,N}\}$
and the domain $\Omega=[1,N+1]$. Then $m\geq N$.
\end{proposition}

\begin{proof}
To this end, we assume $m<N$. Then there exists $j_0\in [1,N]\cap \Z$, such that
 $$
 [j_0,j_0+1)\cap \{x_1,\ldots,x_m\}=\emptyset.
 $$
Thus,  $B_{1,j_0}(x_t)=0$ for $t=1,\ldots, m$. Therefore,
$$
B_{1,j_0}(x_t)=2B_{1,j_0}(x_t)=0,\quad \text{ for any } 1\leq t \leq m.
$$
Note that $B_{1,j_0}, 2B_{1,j_0}\in {\bf U}^1$ (see definition (\ref{eq:definition-U})) and $B_{1,j_0}\neq 2B_{1,j_0}$. Hence, $\{x_1,\ldots,x_m\}$ is not unisolvent.
As a result, $m\geq N$.
\end{proof}

\subsection{Sparse interpolation for Chebyshev systems}

We now turn to Problem 2.  We shall prove that
any $2s$ distinct points in $\Omega$ are unisolvent for the $s$-sparse interpolation provided that any $2s$ functions in  $\{B_j\}_{j=1}^N$ is Chebyshev system.
\begin{theorem}\label{th:chebyshev}
The following properties are equivalent:
\begin{enumerate}[{\rm (i)}]
\item  Suppose that  $f, g\in {\mathbf U}^s$ and
 $ x_1,\ldots,x_{2s}$ are any $2s$ distinct points in $\Omega$.
If
 $$
 f(x_j)=g(x_j),\qquad j=1,\ldots, 2s,
 $$
 then $f\equiv g$.
\item For any index set $T$ with $\# T= 2s$, the function system $\{B_j\}_{j\in T}$ is a Chebyshev system.
\end{enumerate}
\end{theorem}
\begin{proof}
We first show (i) implies (ii). Suppose (ii) is false, namely, there exists an index set  $T=\{j_1,\ldots,j_{2s}\}\subset [1,N]$, and a
set of distinct points $\{x_1,\ldots,x_{2s}\}\subset \Omega$ such that
$${\det}[B_j(x_k)]_{k=1,\ldots,2s, j\in T}\,\,=\,\,0,$$
 which implies that there exists $[c_{j_1},\ldots,c_{j_{2s}}]\neq 0$
such that
\begin{equation}\label{eq:xiangdeng}
\sum_{t=1}^{2s}c_{j_t}B_{j_t}(x_k)=0, \qquad k=1,\ldots,2s.
\end{equation}
Set
$$
f=\sum_{t=1}^{s}c_{j_t}B_{j_t}, \qquad g=-\sum_{t=s+1}^{2s}c_{j_t}B_{j_t}.
$$
Statement (i) implies that any $2s$ functions in $\{B_j\}_{j=1}^N$ are linearly independent and hence
$f\not\equiv g$.
Also, according to (\ref{eq:xiangdeng}) we have
$$
f(x_j)\,\,=\,\,g(x_j),\qquad j=1,\ldots,2s,
$$
which implies $f\equiv g$. This leads to a contradiction.

We next show (ii) implies (i). We take $f, g\in {\mathbf U}^s$ with
$$
f(x_k)=g(x_k),\qquad k=1,\ldots,2s.
$$
To this end, we suppose that $f\not\equiv g$.
We write $f,g$ in the form of
\begin{equation*}
f=\sum_{j\in T_0}c_jB_j,\qquad g=\sum_{j\in T_1}c_jB_j
\end{equation*}
with $\#T_0=\#T_1=s$. Without loss of generality, we suppose $T_0\cap T_1=\emptyset$. Then $f(x_k)=g(x_k)$ implies that
\begin{equation}\label{eq:min}
\sum_{j\in T_0}c_jB_j(x_k)-\sum_{j\in T_1}c_jB_j(x_k)=0, \qquad k=1,\ldots, 2s.
\end{equation}
The fact that ${\rm det}[B_j(x_k)]\neq 0$ implies that the solution to (\ref{eq:min}) is 0, and this contradicts to $f\not\equiv  g$.
\end{proof}

\begin{remark}
We list in the following some function systems which satisfy  {\rm(ii)} in Theorem \ref{th:chebyshev}.
\begin{enumerate}
\item We take $B_j(x)=e^{\lambda_j x}, j=1,\ldots,N$ and $\Omega=\R$, where $\lambda_1<\lambda_2<\cdots<\lambda_N$.
\item We take $B_j(x)=x^{\lambda_j}, j=1,\ldots,N$ and $\Omega=(0, \infty)$, where $\lambda_1<\lambda_2<\cdots<\lambda_N$.
\item We take $B_j(x)=(x+\lambda_j)^{-1}, j=1,\ldots,N$ and $\Omega=(0, \infty)$, where $0\leq \lambda_1<\lambda_2<\cdots<\lambda_N$.
\end{enumerate}
\end{remark}

\section{Deterministic sampling for sparse high dimensional Chebyshev polynomials}
Throughout the rest of this paper, we consider sparse interpolation for the high dimensional Chebyshev polynomials on the domain $\Omega=[-1,1]^d$. Recall that we use $\phi_{n_j}(x_j)$ to denote the one dimensional Chebyshev polynomial in variable $x_j$ with degree $n_j$ and
\begin{align*}
 \mathbf{\Phi}_\mathbf{n}(\mathbf{x})= \prod_{j=1}^d \phi_{n_j}(x_j).
\end{align*}
Note that, for any finite index set $\Lambda\subset \N^d$, the functions $\{{\bf \Phi}_{\bf n}\}_{{\bf n}\in \Lambda}$ are
strongly linearly independent on $[-1,1]^d$. Then, Theorem \ref{th:uni} implies that there are $2s$ points $\{x_1,\ldots,x_{2s}\}\subset [-1,1]^d$
such that they are unisolvent. A well-known result in approximation theory is that there is no Chebyshev systems of continuous functions on $[-1,1]^d$ provided $d\geq 2$ \cite{AT}. Therefore, based on Theorem \ref{th:chebyshev}, one cannot hope that an arbitrary pointset with $2s$ distinct points on  $ [-1,1]^d$ is unisolvent when $d\geq 2$. So, we shall identify $m\geq 2s$ points  $\{x_1,\ldots,x_m\}\subset [-1,1]^d$ such that one can recover $f\in {\mathbf U}^s$ efficiently from $\{(x_j,f(x_j))\}_{j=1}^m.$

We next present a deterministic point set for the $s$-sparse interpolation in high dimensional Chebyshev polynomial spaces. Suppose that $M$ is a prime number. We  define the point set $\Theta_M\subset [-1,1]^d$ as follows
\begin{align*}
\Theta_M:=\left\{{\bf x}_j=\cos( {\bf p}_j) : {\bf p}_j=2\pi\left(j,j^2,\ldots,j^d\right)/M, \,\,\, j=0,\ldots, \lfloor M/2 \rfloor \right\}.
\end{align*}
To state conveniently, throughout the rest of this paper, we set
$$
m:=\# \Theta_M=\lfloor M/2 \rfloor+1.
$$
The next lemma shows the reason why one takes $0\leq j\leq\lfloor M/2 \rfloor$ instead of $0 \leq j \leq M$ in the definition of $\Theta_M$.
\begin{lemma}\label{le:mid}
For any integer $M$ and $m=\lfloor M/2 \rfloor+1$, we have
$$
\cos(2\pi j^k/M)=\cos(2\pi (M-j)^k/M), \quad \text{ for all } \, k\in \mathbb{N} \,\, \text{and} \,\,  0\leq j\leq m-1.
$$
\end{lemma}
\begin{proof}
We first consider the case with $k$ being an even number.
Note that
$$
(M-j)^k=\sum_{t=0}^{k-1}{k\choose t}M^{k-t}(-j)^t+j^k,
$$
and hence
$$
(M-j)^k-j^k\,\,=\,\,0   \mod M,
$$
which implies
$$
\cos(2\pi j^k/M)=\cos(2\pi (M-j)^k/M).
$$
We next turn to the case with $k$ being an odd number. Using similar derivations, we have
$$
(M-j)^k=\sum_{t=0}^{k-1}M^{k-t}(-j)^t-j^k.
$$
Then
$$
(M-j)^k+j^k\,\,=\,\,0   \mod M.
$$
Thus we obtain
$$
\cos(2\pi j^k/M)=\cos(2\pi (M-j)^k/M).
$$
This completes the proof.
\end{proof}

\subsection{Spares interpolation in TP Chebyshev polynomial spaces}

We now choose the point set $\Theta_M$ as the interpolation points for the $s$-sparse interpolation with the  function system $\{\mathbf{\Phi}_\mathbf{n}:  \mathbf{n}\in \Lambda_{\bf P}^{q,d} \}$ and  employ  the
$\ell_1$ minimization as the recovery method.
Recall that
${\bf p}_j=2\pi\left(j,j^2,\ldots,j^d\right)/M$.
Note that
$$
\mathbf{\Phi}_\mathbf{n}({\bf x}_j)=\mathbf{C}_\mathbf{n}(\mathbf{p}_j),
$$
where
\begin{align*}
\mathbf{C}_{\mathbf{n}}(\mathbf{p}_j):=\prod_{t=1}^d \cos(2\pi n_tj^t/M).
\end{align*}
Then the interpolation matrix is
\begin{align}\label{eq:TPA}
\mathbf{A_P}\,\,:=\,\,\big[\mathbf{C}_\mathbf{n}(\mathbf{p}_j)\big]_{j=1,\ldots,m; \, {\bf n}\in \Lambda_{{\bf P}}^{q,d}}\,\,\in\,\, \R^{m\times (q+1)^d}.
\end{align}
The following lemma gives an estimation of $\mu(\mathbf{A_P})$.
\begin{lemma}\label{le:MICA}
Suppose that $M\geq \max\{2q+1, (2^d(d-1))^2\}$ is a prime number.
Then
\begin{align*}
\mu(\mathbf{A_P})\leq \frac{1}{\sqrt{M}} \frac{2^d\cdot d}{1-\frac{2^{d}(d-1)}{\sqrt{M}}}.
\end{align*}
\end{lemma}

\begin{proof}
Let us first consider
\begin{align}
\left |\sum_{j=0}^{m-1} \mathbf{C}_{\bf n}(\mathbf{p}_j)\mathbf{C}_{\bf k}(\mathbf{p}_j)\right|,
\end{align}
where ${\bf n}, {\bf k} \in \Lambda_{{\bf P}}^{q,d}$ and ${\bf n}\neq {\bf k}$.
By repeatedly using  the fact $\cos(\alpha)\cos(\beta)=\frac{1}{2}(\cos(\alpha+\beta)+\cos(\alpha-\beta))$,  we have
\begin{eqnarray*}
\mathbf{C}_{\mathbf{n}}({\bf p}_j)\mathbf{C}_{\mathbf{k}}({\bf p}_j)
&=&\prod_{t=1}^d \cos(2\pi n_t j^t/M)\cos(2\pi k_t j^t/M)\\
&=&\frac{1}{2^{2d-1}}\sum_{\epsilon\in \{-1,1\}^{2d-1}} \cos(\mathbf{t}(\epsilon,\mathbf{p}_j)),
\end{eqnarray*}
where
\begin{align*}
\mathbf{t}(\epsilon,\mathbf{p}_j)=2\pi((n_1+\epsilon_1k_1)j+(\epsilon_2n_2+\epsilon_3k_2)j^2+ \cdots+ (\epsilon_{2d-2}n_d+\epsilon_{2d-1}k_d)j^d)/M.
\end{align*}
 Note that there are totally $2^{2d-1}$ possible $\epsilon$.
The Weil's theorem implies that for a fixed $\epsilon\in \{-1,1\}^d$, there holds
\begin{align}
\left|\sum_{j=1}^M \cos(\mathbf{t}(\epsilon,\mathbf{p}_j)) \right| \leq \left|\sum_{j=1}^M \exp({\rm i}\mathbf{t}(\epsilon,\mathbf{p}_j)) \right| \leq (d-1)\sqrt{M}.
\end{align}
Based on Lemma \ref{le:mid}, we have
\begin{align*}
2\sum_{j=0}^{m-1}\cos(\mathbf{t}(\epsilon,\mathbf{p}_j)) - 1=\sum_{j=1}^M \cos(\mathbf{t}(\epsilon,\mathbf{p}_j)).
\end{align*}
Thus, we have
\begin{align}
\left|\sum_{j=0}^{m-1}\cos(\mathbf{t}(\epsilon,\mathbf{p}_j)) \right|\leq \frac{(d-1)\sqrt{M}+1}{2},
\end{align}
which implies
\begin{align}\label{eq:mk}
\left|\sum_{j=0}^{m-1} \mathbf{C}_{\bf n}(\mathbf{p}_j)\mathbf{C}_{\bf k}(\mathbf{p}_j)\right|\leq \frac{(d-1)\sqrt{M}+1}{2}.
\end{align}
Let us now consider the column norm of $\mathbf{A_P}$. Using repeatedly the fact $\cos(2\alpha)=2\cos^2(\alpha)-1$ and the similar procedure above, we can obtain that
\begin{align*}
\sum_{j=1}^M \left|\mathbf{C}_{\mathbf{n}}({\bf p}_j)\right|^2\geq \frac{M}{2^d}-{(d-1)\sqrt{M}}.
\end{align*}
Based on Lemma \ref{le:mid}, we have
\begin{align}\label{eq:norm}
\sum_{j=0}^{m-1} \left|\mathbf{C}_{\mathbf{n}}({\bf p}_j)\right|^2\geq \frac{M}{2^{d+1}}-\frac{(d-1)\sqrt{M}}{2}.
\end{align}
Then the desired  result can be obtained by combining (\ref{eq:mk}) and (\ref{eq:norm}).
\end{proof}
We now arrive at one of the main results of this paper, which shows that one can use $\ell_1$ minimization to   recover
$f\in {\mathbf U}^s(\{\mathbf{\Phi}_\mathbf{n}:  \mathbf{n}\in \Lambda_{\bf P}^{q,d} \})$ from the function values
of the points $\mathbf{x}_j\in \Theta_M$.

\begin{theorem}\label{th:TP}
Suppose that $M \geq \max\{2q + 1, 9\cdot 4^{d}\cdot d^2\cdot s^2\}$ is a prime number.
Let
  $$
  f =\sum_{{\mathbf n}\in \Lambda_{\bf P}^{q,d} } c_{{\mathbf n}}\mathbf{\Phi}_\mathbf{n},
  $$
and assume that $\c^\#$ is given by the $\ell_1$-minimization problem (\ref{eq:P1}) with the matrix $\mathbf{A}:=\mathbf{A_P}\cdot \mathbf{C}$ and
the vector $\mathbf{b}:=(f(x_1),\ldots,f(x_m))^\top$, where
$x_j\in \Theta_M$,  $m=\lfloor M/2 \rfloor+1$ and ${\mathbf C}$ is a diagonal matrix such that  the  columns of $\mathbf{A}$
are standardized to have unit $\ell_2$ norm. Then we have
\begin{equation}\label{eq:21}
\|\mathbf{c}^\#-\mathbf{c}\|_2\,\,\lesssim\,\,   \frac{\sigma_{s,1}(\mathbf{c})}{\sqrt{s}}.
\end{equation}
\end{theorem}
\begin{proof}
 According to the results in  \cite{MR1, MR2}, the matrix ${\mathbf A}$
 satisfies $s$-order RIP property with RIP constant
 $$
 \delta_s\leq (s-1)\mu({\mathbf A})=(s-1)\mu({\mathbf A_P}).
 $$
The result in \cite{CZ} shows that (\ref{eq:21}) holds provided $\delta_s<1/3$.
According to Lemma \ref{le:MICA}, if $M \geq \max\{2q + 1,  9\cdot 4^{d}\cdot d^2\cdot s^2\}$, then
 $$
 \delta_s\leq (s-1)\cdot \mu({\mathbf A_P}) <1/3,
 $$
  which implies  (\ref{eq:21}).
\end{proof}
\begin{remark}
If we suppose that $f\in {\mathbf U}^s$ in Theorem \ref{th:TP}, i.e., the vector $\c$  is $s$-sparse,  then
Theorem \ref{th:TP} implies that $(P_1)$ can recover the $s$-sparse function $f$ exactly from the function values on
the point set $\Theta_M$ provided  the number of interpolation points  $m \geq \max\{q , 9/2\cdot 4^{d}\cdot d^2\cdot s^2\}+1$ and $M$ is
a prime number.
\end{remark}
\subsection{Sparse interpolation in TD Chebyshev polynomial spaces}

In uncertainty quantification, the dimension $d$ is often determined by the number of random parameters and can be very large. One often encounters practical stochastic problems with the dimension $d$ on the order of hundreds \cite{Fabio, XiuL1}. In such cases, one can not afford to construct high-degree TP polynomial approximations. We now consider the sparse interpolation in the TD spaces. In other words, we study the $s$-sparse interpolation with  the function system is $\{{\bf \Phi}_{\bf n}:{{\bf n}\in \Lambda_{\bf D}^{q,d}}\}$.

The use of TD spaces is promising for cases where the dimensionality is high such that one can not afford to construct high-degree polynomial approximations, that is, we usually consider the cases where $d\gg q$. Similar with before,
the interpolation matrix is
\begin{align}\label{eq:A'}
\mathbf{A_D}:=\big[\mathbf{C}_\mathbf{n}(\mathbf{p}_j)\big]_{j=1,\ldots,m; \, {\bf n}\in \Lambda_{{\bf D}}^{q,d}}.
\end{align}
We have the following lemma:
\begin{lemma}\label{le:TD}
Suppose that $d\geq q$ and $M\geq \max\{2q+1, (2^q(d-1))^2\}$ is a prime number.
Then
\begin{align*}
\mu(\mathbf{A_D})\leq \frac{1}{\sqrt{M}} \frac{2^q\cdot d}{1-\frac{2^{q}(d-1)}{\sqrt{M}}}.
\end{align*}
\end{lemma}
\begin{proof}

Take ${\bf n}, {\bf k} \in \Lambda_{{\bf D}}^{q,d}, {\bf n}\neq {\bf k}$ and consider
\begin{align*}
\left|\sum_{j=0}^{m-1} \mathbf{C}_{\bf n}(\mathbf{p}_j)\mathbf{C}_{\bf k}(\mathbf{p}_j)\right|.
\end{align*}
Using a similar method as the proof of Lemma \ref{le:MICA}, we have
\begin{align}\label{eq:inn1}
\left|\sum_{j=0}^{m-1} \mathbf{C}_{\bf n}(\mathbf{p}_j)\mathbf{C}_{\bf k}(\mathbf{p}_j)\right|\leq \frac{(d-1)\sqrt{M}+1}{2}.
\end{align}
Let us now investigate the column norm of $\mathbf{A_D},$ e.g.
\begin{align*}
\sum_{j=0}^{m-1} \left|\mathbf{C}_{\mathbf{n}}({\bf p}_j)\right|^2.
\end{align*}
Note that $|\mathbf{n}|\leq q$ provided that ${\bf  n}\in \Lambda_{{\bf D}}^{q,d}$.  Then $\|{\bf n}\|_0\leq q$  provided that $d\geq q$.
 This  gives the following estimate
\begin{align}\label{eq:norm1}
\sum_{j=0}^{m-1} \left|\mathbf{C}_{\mathbf{n}}({\bf p}_j)\right|^2\geq \frac{M}{2^{q+1}}-\frac{(d-1)\sqrt{M}}{2}.
\end{align}
Combining (\ref{eq:inn1}) and (\ref{eq:norm1}), we  complete the proof.
\end{proof}
Note that $4^q> 2q+1$. Using Lemma \ref{le:TD} and a similar method in Theorem \ref{th:TP}, we have
\begin{theorem}\label{th:TD}
  Suppose that $M \geq 9\cdot 4^{q}\cdot d^2\cdot s^2$ is a prime number.
 Let
  $$
  f =\sum_{{\mathbf n}\in \Lambda_{\bf D}^{q,d} } c_{{\mathbf n}}\mathbf{\Phi}_\mathbf{n}.
  $$
Assume that $\c^\#$ is given by the $\ell_1$-minimization problem (\ref{eq:P1}) with the matrix $\mathbf{A}:=\mathbf{A_D}\cdot \mathbf{C}$  and
the vector $\mathbf{b}=(f(x_1),\ldots,f(x_m))^\top$, where
$x_j\in \Theta_M$,  $m=\lfloor M/2 \rfloor+1$ and ${\mathbf C}$ is a diagonal matrix such that  the  columns of $\mathbf{A}$
are standardized to have unit $\ell_2$ norm. Then we have
\begin{equation*}
\|\mathbf{c}^\#-\mathbf{c}\|_2\,\,\lesssim\,\,   \frac{\sigma_{s,1}(\mathbf{c})}{\sqrt{s}}.
\end{equation*}
\end{theorem}
\begin{remark}
Similar with before, if we suppose that $f\in {\mathbf U}^s$ in Theorem \ref{th:TD}, i.e., the vector $\c$  is $s$-sparse,  then
Theorem \ref{th:TD} implies that $(P_1)$ can recover the $s$-sparse function $f$ exactly from the function values on
the point set $\Theta_M$ provided  the number of interpolation points  $m \geq 9/2\cdot 4^{q}\cdot d^2\cdot s^2+1$ and $M$ is
a prime number.
\end{remark}
\begin{remark}
When $M$ is divided by $4$, the set of the first entry of ${\bf p}_j$ is Chebyshev nodes set where $j$ runs over all  odd integers in
$[1,M/2]$. And hence, the points ${\bf p}_j$ can be considered as an extension of Chebyshev nodes. We hope to investigate multivariate
Lagrange Interpolation on the points ${\bf p}_j$ in future work (see \cite{yuanxu}).
\end{remark}

\section{Numerical examples}
In this section we  make numerical experiments to compare the performance of the determinant points $\Theta_M$ and that
of the random points. The random interpolation points are chosen based on the continuous probability model, i.e. $x_1,\ldots,x_m$ are independent random
variable having the uniform distribution on $[-1,1]^d$.
 Given the function system $\{B_j\}_{j\in \Lambda}$, the support set of $f\in {\bf U}^s$ is drawn from the uniform
distribution over the set of all subset of $\Lambda$ of size $s$. The non-zero coefficients of $f$  have the Gaussian distribution with mean zero
and standard deviation one. To solve the $\ell_1$ minimization, we  employ the available tools SPGL1 from \cite{L1matlab} that was implemented in the MATLAB.
We repeat the experiment $100$ times for each fixed sparsity $s$ and calculate the success rate.
We will conduct two groups of tests, namely, the TP Chebyshev spaces and the TD Chebyshev spaces.
\begin{figure}[t]
\begin{center}
\includegraphics[width=0.48\textwidth]{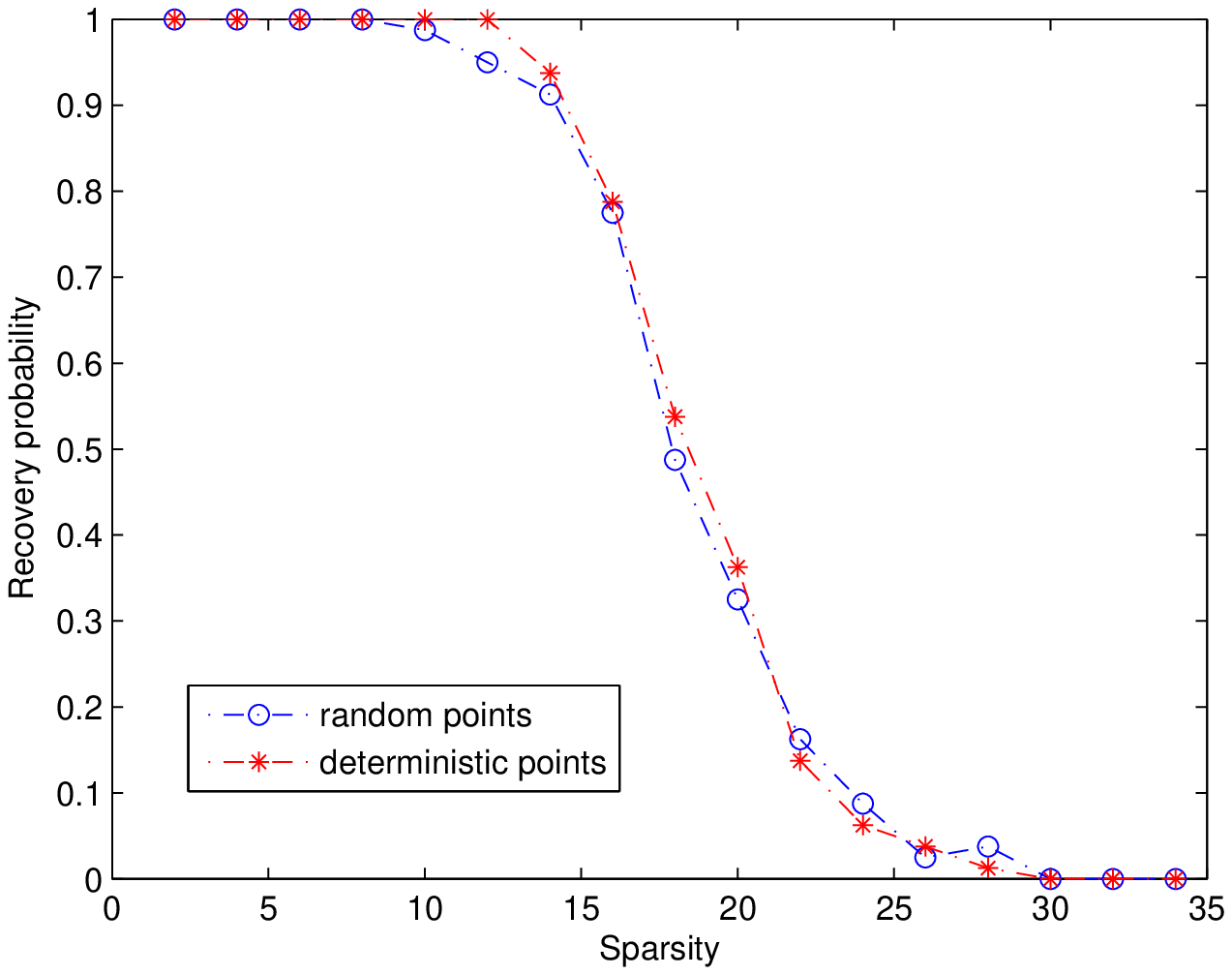}
\includegraphics[width=0.48\textwidth]{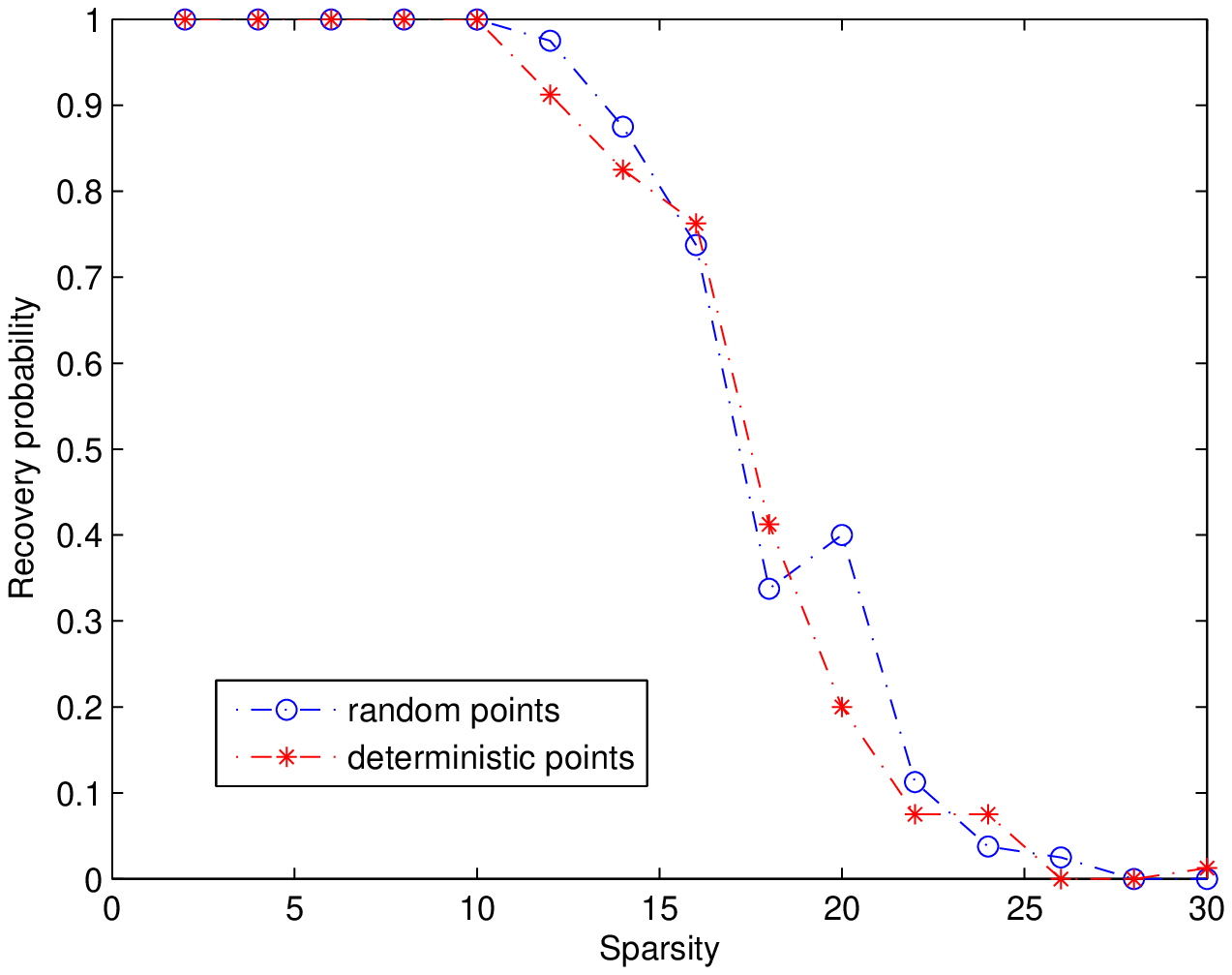}
\end{center}
\caption{
Numerical results for  the comparison of the random sampling and the deterministic sampling of the sparse interpolation
 in TP spaces. Left: $d=2,$ $q=9,$ $m=49,$ and Right: $d=3,$ $q=6,$ $m=69.$}
\end{figure}
\subsection{Tests for the TP Chebyshev spaces}

We first choose the function system as
$$
\{\mathbf{\Phi}_\mathbf{n}:  \mathbf{n}\in \Lambda_{\bf P}^{q,d} \}.
$$
As we discussed before, TP spaces are not frequently used in real applications due to the curse of dimensionality. Thus, we consider low dimensional cases of $d=2$ and $d=3,$   and we also choose the degree $q$ as $9$ and $6$, respectively.
 We remark that the parameters $d$ and $q$ chosen in ours numerical examples bear no special meaning, as the results from other parameters demonstrate similar behavior.  The left graph in Fig. 1 depicts the success rate when $d=2$, $q=9$ and $m=49$ points are used, while the right graph
 shows the success rate for  $d=3, q=6,$ and $m=69$.
The numerical results show that the performance of the deterministic points is similar with that of the random points.

\begin{figure}[t]
\begin{center}
\includegraphics[width=0.48\textwidth]{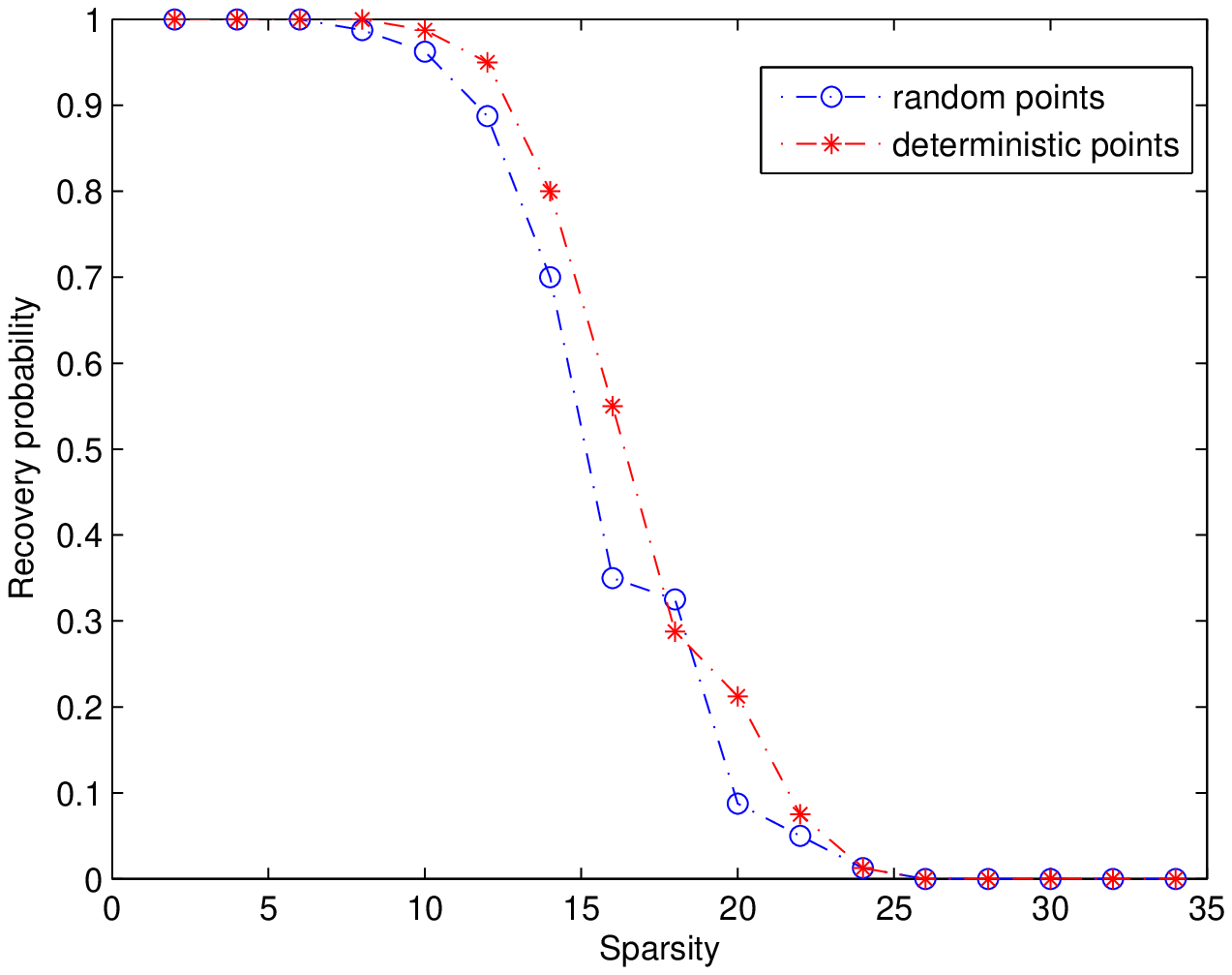}
\includegraphics[width=0.48\textwidth]{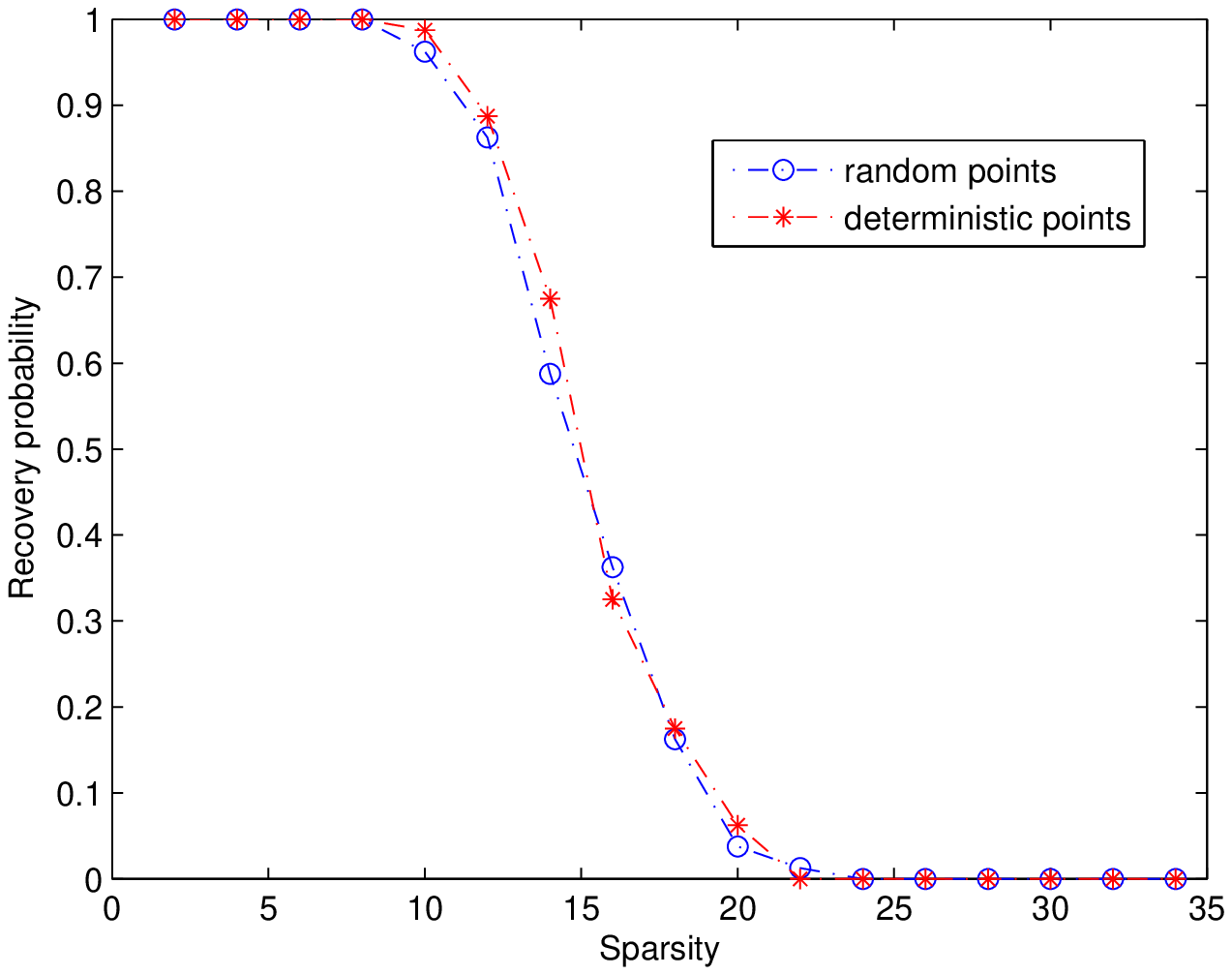}
\end{center}
\caption{
Numerical results for the  comparison of the random sampling and the deterministic sampling of the sparse interpolation
 in TD spaces ($m=69$). Left: $d=10, q=3,$ and right: $d=30,q=2.$}
\end{figure}

\subsection{Tests for the TD Chebyshev spaces}
Now we choose the function system as
 $$
\{\mathbf{\Phi}_\mathbf{n}:  \mathbf{n}\in \Lambda_{\bf D}^{q,d} \}
$$
and test the recovery properties in the TD spaces, which are very useful when dealing with high dimensional problems. In this part, we will consider high dimensional cases with $d=10$ and $d=30.$ The numerical treatment is the same as in TP spaces and the recovery results are demonstrated in Fig. 2. The right plot is for $d=30, q=2,$ and $m=69$, while the left plot is for $d=10, q=3$ and $m=69$. Again, the performance of the deterministic points is comparable with that of the random points.

\section*{Acknowledgment}
Z. Xu is supported by the National Natural Science Foundation of China (11171336) and by the Funds for Creative
       Research Groups of China (Grant No. 11021101).  T. Zhou
is supported by the National Natural Science Foundation of China (No.91130003 and No.11201461).

\end{document}